\documentclass[11pt]{amsart}

\usepackage{paracol, ragged2e} 

\usepackage[vietnamese, english]{babel} 
\usepackage[utf8]{inputenc}
\usepackage{cmbright}
\usepackage[T1]{fontenc}
\usepackage{lmodern} 
\usepackage{soul,xfrac}
\usepackage{mathtools, textcomp}
\mathtoolsset{showonlyrefs} 
\usepackage[shortlabels]{enumitem}
\usepackage{cite} 
\usepackage{amsmath,amssymb,amsthm,mathtools,color}
\usepackage{xcolor}
\usepackage{hyperref}
\definecolor{darkgreen}{rgb}{0,0.4,0}
\definecolor{BrickRed}{rgb}{0.65,0.08,0}
\hypersetup{colorlinks=true,linkcolor=blue,citecolor=red,filecolor=BrickRed,urlcolor=darkgreen}

\usepackage{graphicx} 
\graphicspath{{pics/}}

\usepackage[bottom]{footmisc}

\newtheorem{thm}{Theorem}

\newtheorem{rmk}{Remark}

\newtheorem{prop}[]{Proposition}
\newtheorem{conj}{Conjecture}
\newtheorem{lemma}[]{Lemma}

\numberwithin{equation}{section}

\def\N{\mathbb{N}}

\def\F{\mathbb{F}}

\def\C{\mathbb{C}}

\def\Kl{\mathrm{Kloos}}
\def\ainFq*{a\in \F_q^*}

\def\rinN*{r \in \N^*}
\def\Re{\mathrm{Re}}

\def\ninN*{n \in \N^*}

\def\SGA412{\text{SGA} 4 \sfrac{1}{2}}

\newtheorem{conjecturex}{Conjecture}
\newtheorem{thmx}{Theorem}

\def\aa{\mathfrak{a}}
\def\Flat{\displaystyle\sideset{}{^\flat}}
\def\Sharp{\displaystyle\sideset{}{^\sharp}}
\def\Star{\displaystyle\sideset{}{^*}}
\def\Prime{\displaystyle\sideset{}{^\prime}}

\addto\captionsenglish{
	
}

\begin{document}
\title{A note on $\theta_2$}

\author{David T. Nguyen}

\address{Department of Mathematics and Statistics, Queen's University, Jeffery Hall, 48 University Ave, Kingston, Ontario, K7L-3N6, Canada}
\email{d.nguyen@queensu.ca}

\dedicatory{\bigskip\bigskip
		Dedicated to Professor Yitang Zhang, with respect and admiration}

\maketitle

\begin{abstract}
	It is a classic result of Selberg in the 1950's that $\theta_2 = 2/3$, where $\theta_2$ is the level of distribution of the divisor function in arithmetic progressions (defined more precisely below). Selberg applies this estimate, together with his $\Lambda^2$ sieve, to prove weak forms of the binary Goldbach and twin prime conjectures. 
	
	In this note, we give an unconditional proof of Selberg's first result with smooth weights for prime moduli, and put forth a hypothesis via subconvexity to bootstrap to level $4/5$. Contingent outcomes of our proposal is an improvement of Selberg's second result on approximations to twin primes and in lowering bounds on gaps between primes.
\end{abstract}

\section{Introduction}

Let $\tau_2(n)$ denote the usual divisor function. For integers $m$ and $n$, let $S(m,n; q)$ denote the Kloosterman sums
\begin{equation}
	S(m,n; q)
	=
	\sum_{h \overline{h} \equiv 1 \pmod q}
	e^{2 \pi i \frac{mh + n \overline{h}}{q}}
\end{equation}
where $h$ runs over a complete set of residue classes coprime to $q$. One knows that
\begin{equation} \label{eq:WeilBound}
	|S(m,n; q)|
	\le (m,n, q)^{1/2}
	q^{1/2} 
	\tau_2(q),
\end{equation}
where $(m,n, q)$ denotes the G.C.D. of $m$, $n$, and $q$. As remarked by Selberg \footnote{see Footnote 64 on page 236 in \cite{SelbergVol2}}, the deepest part of \eqref{eq:WeilBound} is when $q=p$ is prime and is due to Weil \cite{Weil1948} (1948) from his resolution of the Weil Conjectures for curves over finite fields, with the remaining cases follow from the earlier result of Salie \cite{Salie1931} (1931) for prime powers $q= p^\gamma, \gamma > 1$, and multiplicative properties of $S(m,n; q)$. We will refer to \eqref{eq:WeilBound} as Weil's upper bound for complete Kloosterman sums, or Weil's bound for Kloosterman sums for short, or, simply, Weil's bound.

A classic unpublished result of Selberg from the 1950's states that Weil's upper bound for complete Kloosterman sums yields
\begin{thmx}[Selberg, 1950's] \label{thm:S}
	For $(a,q) = 1$ and $q < X^2$ we have for any $\epsilon>0$, that
	\begin{equation} \label{eq:S}
		\sum_{\substack{n \le X\\ n\equiv a \pmod q}}
		\tau_2(n)
		=
		X
		\frac{\varphi(q)}{q^2}
		\left(
		\log X + 2\gamma
		+ 2 \sum_{p \mid q} \frac{\log p}{p-1}
		\right)^2
		+
		O\left(
		q^{-1/4} X^{1/2+\epsilon}
		\right).
	\end{equation}
\end{thmx}

Hooley had obtained in \cite[Lemma C, p. 404]{Hooley1957} a related result but only for the leading order main term with a $\log X$ saving in the error term. The error term above is $\ll q^{-1} X^{1-\delta}$ when
\begin{equation} \label{eq:S2}
	q \ll X^{\frac{2}{3} - 2 \delta}.
\end{equation}
The number $2/3$ in the exponent of the above is called level of distribution of $\tau_2$, denoted by $\theta_2$. More precisely, we say that the divisor function has a level of distribution $\theta_2$ if for any $\epsilon>0$, there exists $\delta = \delta(\epsilon) >0$ such that
\begin{equation}
	\sum_{\substack{n\le X\\ n\equiv a \pmod q}}
	\tau_2(n)
	=
	\frac{X}{\varphi(q)}
	P_2(X, q)
	+
	O_\delta \left( \frac{X^{1-\delta}}{\varphi(q)} \right)
\end{equation}
uniformly for all $(a,q)=1$ and $q\le X^{\theta_2 - \epsilon}$, where $P_2(X, q)$ is a linear polynomial in $\log X$ with coefficients depending on $q$. We note the bound $1/\varphi(q) \ll (\log\log q)/q$. Thus, Selberg's \eqref{eq:S} gives $\theta_2 = 2/3$.

The bound \eqref{eq:S} has applications towards a weak form of the twin prime and Goldbach conjectures (see Section \ref{sec:why?} below), the classical correlation sum $\sum_{n\le X} \tau_2(n) \tau_2(n+h)$, and, hence, to moments of $L$-functions. Thus, extending \eqref{eq:S} beyond the two-thirds barrier is an important problem in number theory.

There are two published proofs of \eqref{eq:S} with power-saving error terms, to our knowledge, that we now briefly review. The first proof that appeared in print seems to be from Heath-Brown in his 1979 paper \textit{The Fourth Power Moment of the Riemann Zeta Function} \cite{HB1979}. More precisely, Heath-Brown proved in \cite[Corollary 1, p. 409]{HB1979} that for any $\epsilon>0$, and any $a,q\ge 1$ (not necessarily coprime), one has
\begin{equation} \label{eq:HB}
	\sum_{\substack{n \le X\\ n\equiv a \pmod q}}
	\tau_2(n)
	=
	X
	q^{-2}
	\left(
	A(q, a) (\log x/q^2) + 2 \gamma - 1
	+ 2 B(q, a)
	\right)
	\ll
	X^{1/3 + \epsilon},
\end{equation}
uniformly for $1 \le q \le X^{2/3}$, where
\begin{equation}
	A(q,a)
	=
	\sum_{d \mid (q,a)}
	\sum_{\delta \mid q/d}
	d \delta \mu(q/d \delta)
\end{equation}
and
\begin{equation}
	B(q,a)
	=
	\sum_{d \mid (q,a)}
	\sum_{\delta \mid q/d}
	d \delta
	\mu(q/d \delta)
	\log \delta.
\end{equation}
The main terms in \eqref{eq:HB} and \eqref{eq:S} agree when $(a,q)=1$. Heath-Brown's proof uses Voronoi summation and follows Titchmarsh's treatment \cite[\textsection 12.4]{TitchmarschBook} for the classical Dirichlet divisor problem corresponding to the case $q=1$ , but with the function
\begin{equation} \label{eq:F}
	F(s; q, a)
	=
	\sum_{n \equiv a \pmod q}
	\frac{\tau_2(n)}{n^s}
\end{equation}
instead of $\zeta^2(s)$, and his proof is also powered by Weil's bound for Kloosterman sums. Weil's bound is used to bound the level aspect of the function $F(s; q, a)$ on the left half plane, and yields (c.f. \cite[eq. (57)]{HB1979})
\begin{equation}
	F(s; q, a)
	\ll
	(1+ |t|)^{1- 2\sigma}
	q^{\frac{1}{2} - 2\sigma + \epsilon},\quad
	(\sigma <0).
\end{equation}
Heath-Brown's motivation for proving \eqref{eq:HB} was to obtain the finer full main term with a power-saving error term for the forth moment of the Riemann zeta function on the critical line. 

The second published proof of \eqref{eq:S} is from Selberg himself but appeared latter.  In his \textit{Lecture on sieves} which can be found as the last paper in the second volume of his collected works, Selberg obtain in \cite[Lemma 18, p. 235]{SelbergVol2} the bound quoted in Theorem \ref{thm:S}. It seems that these lecture notes were finished in late 1990 (see \cite{SelbergVol2}[Afterword, p. 252]). Selberg's proof is based on his observations that
\begin{equation}
	\frac{1}{\delta}
	\int_{X e^{-\delta}}^X 
	\sum_{\substack{n\le t\\ n \equiv a \pmod q}}
	\tau_2(n)
	\frac{dt}{t}
	\le
	\sum_{\substack{n\le X\\ n \equiv a \pmod q}}
	\tau_2(n)
	\le
	\frac{1}{\delta}
	\int_{X}^{X e^{\delta}} 
	\sum_{\substack{n\le t\\ n \equiv a \pmod q}}
	\tau_2(n)
	\frac{dt}{t},
	(0 < \delta <1),
\end{equation}
and
\begin{align}
	\int_1^X 
	\sum_{\substack{n\le t\\ n \equiv a \pmod q}}
	\tau_2(n)
	\frac{dt}{t}
	&=
	\sum_{\substack{n\le X\\ n\equiv a (q)}}
	\tau_2(n)
	\log\frac{X}{n}
	\\&=
	\frac{1}{k^2}
	\sum_{-q/2 < m,n \le q/2}
	S(m, na; q)
	\int_1^X
	S_t(m/q)
	S_{X/t}(n/q)
	\frac{dt}{t},
\end{align}
where
\begin{equation}
	S_t(\alpha)
	=
	\sum_{0< \nu <t}
	e^{2 \pi i \nu \alpha},
	\quad
	(|\alpha| < \pi /2 ),
\end{equation}
and is also powered by Weil's bound.

\subsection{Why do people study $\theta_2$?} \label{sec:why?}
As touched upon earlier, one of Heath-Brown's motivations for proving \eqref{eq:HB} was application to moments of the Riemann zeta functions on the critical line, which are approximations to the still unproven Lindel\"of Hypothesis for $\zeta(s)$. As for Selberg, one of his motivations to proving Theorem \ref{thm:S} was applications of his $\Lambda^2$ sieve towards the twin primes and Goldbach problem, which we now provide further details. 

Selberg applied the bound \eqref{eq:S} as a crucial technical ingredient to obtained in \cite[eq. (23.38), p. 244]{SelbergVol2} the following weak form of the twin prime problem.
\begin{thmx}[Selberg, 1990] \label{thm:S2}
	We have, for $X$ sufficiently large,
	\begin{equation} \label{eq:SelbergTP}
		\# \left\{
		n\le X|
		\text{one of $n$ and $n+2$ has at most 2, the other at most 3 primes factors}
		\right\}
		\gg
		\frac{X}{\log^2 X}.
	\end{equation}
\end{thmx}
The number of prime factors 2 and 3 from the above result depend on the level $\theta_2$ in the following way. The $\Lambda^2$ sieve applied to the twin prime problem gives an inequality of the form
\begin{equation} \label{eq:selberg}
	\frac{4}{3} (2^{\nu(n)} )
	+ \frac{2}{3}  (2^{\nu(n+2)})
	< \lambda_2 + \epsilon
\end{equation}
where
\begin{equation} \label{eq:lambda}
	\lambda_2 = \frac{8}{\theta_2} + 2,
\end{equation}
in which \eqref{eq:selberg} holds for $\gg X/\log^2 X$ number of $n \le X$, where $\nu(n)$ denotes the number of prime factors of $n$ counted with multiplicities. For Selberg's $\theta_2 = 2/3$, \eqref{eq:lambda} yields
\begin{equation} \label{eq:14}
	\lambda_2 = 14,
\end{equation}
which implies either
\begin{equation}
	\{ \nu(n) \le 2 \text{ and } \nu_2(n+2) \le 3 \}
	\quad
	\text{or}
	\quad
	\{ \nu(n) \le 3 \text{ and } \nu_2(n+2) \le 2 \},
\end{equation}
as both 
\begin{equation} \label{eq:1237}
	\frac{4}{3}(2^2) + \frac{2}{3}(2^3)
	= \frac{32}{3}
	\approx 10.67
	\quad \text{and} \quad
	\frac{4}{3}(2^3) + \frac{2}{3}(2^2)
	=
	\frac{40}{3}
	\approx 13.33
\end{equation}
are less than 14. To improve on this result of Selberg and eliminate the possibility $\nu(n) = 3$ and $\nu(n+2) = 2$, the value of $\lambda_2$ needs to be brought down to a number smaller than $40/3 \approx 13.33$. Our method here could be used to sharpen this constant $\lambda_2$ from $14$ to $12$, thus eliminates this possibility; see the Applications section \ref{sec:apps} following the statement of Theorem \ref{thm:beatsSelberg} below for more on this.

In \cite{NguyenCorr}, the author showed numerically square-root cancellation in the error term of the classical shifted convolution
\begin{equation} \label{eq:1032}
	\sum_{n \le X}
	\tau_2(n) \tau_2(n+1)
\end{equation}
using the the full asymptotic for
\begin{equation} \label{eq:1033}
	\sum_{\substack{n\le X\\ n\equiv a \pmod q}}
	\tau_2(n).
\end{equation}
A modified form of the correlation \eqref{eq:1032} for the 3-fold divisor function is studied in \cite{NguyenAP2} and found to adequately serve some applications. In this note, we derive an explicit error term for a smoothed version of \eqref{eq:1033} for prime moduli, which, in particular, gives another proof of the known level $\theta_2 =2/3$ (see Theorem \ref{thm:1}), and propose a conjectural way via a Lindel\"of hypothesis to break through this barrier in Theorem \ref{thm:beatsSelberg}.

Recently, Khan in \cite{Khan2016} (2016) succeeded in breaking through this two-thirds barrier when the moduli is a prime power $p^\gamma$ for odd primes $p$ and any fixed integer $\gamma \ge 7$. In contrast to this result of Khan, in which the condition $\gamma$ is fixed and $p$ sufficiently large is required, there is a reciprocal result of Liu, Shparlinski, Zhang, \cite{LSZ-2018} (2018), brought to our attention by I. Shparlinski, that also broke through this barrier and gave a uniform bound for all moduli of the type $q=p^\gamma$, with $p$ a fixed odd prime and $\gamma$ sufficiently large. The case for prime moduli however remains open.

\bigskip
\noindent
\textit{Acknowledgments.}
The author is extremely grateful to Professor Yitang Zhang for introducing him to this problem of $\theta_2$ in 2017. He would like to respectfully dedicate this work to him. Thanks also to Prof.'s M. Ram Murty for a useful reference; and, Andrew Lewis for a plethora of helpful conversations.

\section{Statement of results and discussion of applications}

Heuristics (see, for instance, \cite[bottom of page 4 to middle of page 5]{NguyenVar}) from the function fields setting \cite{KRRR} where rigorous results are known, suggest that the level of distribution $\theta_2$ could be as large as it could be, which is 1. More precisely, there is the following ``level one conjecture" for $\tau_2$, which we now state with smooth weights.
\begin{conjecturex}[Folklore] \label{conj:levelone}
	Let $q \ge 1$ be a composite number. Let $W(x)$ be a smooth test weight supported on $[1,2]$, such that
	\begin{equation} \label{eq:rapiddecay}
		\hat{W}(\sigma + it) \ll_\ell \frac{1}{(1+ |t|)^\ell}
	\end{equation}
	uniformly for all $|\sigma| \le A$ for any fixed positive $A>0$, for all positive integers $\ell$, where $\hat{W}$ denotes the Mellin transform 
	\begin{equation}
		\hat{W}(s)
		=
		\int_0^\infty
		W(x) x^{s-1} dx.
	\end{equation}
	of $W$. Then, for any $\epsilon>0$, there exists $\delta = \delta(\epsilon) >0$ such that
	\begin{equation} \label{eq:tau2c}
		\sum_{\substack{n\\ n\equiv a \pmod q}}
		\tau_2(n) W\left(\frac{n}{X}\right)
		=
		\frac{1}{\varphi(q)}
		\sum_{\substack{n\\ (n,q) = 1}}
		\tau_2(n) W\left(\frac{n}{X}\right)
		+ O_\delta \left( \frac{X^{1-\delta}}{\varphi(q)} \right)
	\end{equation}
	uniformly for all $(a,q) = 1$ and $q \le X^{\theta_{W, 2} - \epsilon}$, with
	\begin{equation} \label{eq:thetaW2}
		\theta_{W, 2} = 1.
	\end{equation}
\end{conjecturex}

To slightly lighten the notation, we sometimes write $\theta_2$ for $\theta_{W ,2}$. 
Towards Conjecture \ref{conj:levelone}, our first result gives a new unconditional proof of a special case of Selberg's result of level two-thirds.

\begin{thm} \label{thm:1}
	Let $p$ be a prime.	Let $W(x)$ be a smooth test weight supported on $[1,2]$ as in Conjecture \ref{conj:levelone}. Then, for any $\epsilon>0$, there exists $\delta = \delta(\epsilon) >0$ such that
	\begin{equation} \label{eq:tau2}
		\sum_{\substack{n\\ n\equiv a \pmod p}}
		\tau_2(n) W\left(\frac{n}{X}\right)
		=
		\frac{1}{p-1}
		\sum_{\substack{n\\ (n,p) = 1}}
		\tau_2(n) W\left(\frac{n}{X}\right)
		+ O_\delta \left( \frac{X^{1-\delta}}{p} \right)
	\end{equation}
	uniformly for all $(a,p) = 1$ and $p \le X^{\theta_{W,2} - \epsilon}$, with
	\begin{equation} \label{eq:2over3}
		\theta_{W,2} = 2/3.
	\end{equation}
\end{thm}

Our method of proof of Theorem \ref{thm:1} is different from both Heath-Brown's and Selberg's surveyed above, and has the potential of breaking through this 2/3 barrier. In an effort to archive that goal, we propose the following
\begin{conj} \label{conj:LH}
	Let $\epsilon>0$. For $s=1/2 + it$, $q\ge 1$, $(A,q)=1$, we have
	\begin{equation} \label{eq:newbeef}
		\sum_{\substack{n,m=1\\ (mn,q)=1}}^\infty
		\frac{e_q (A n \overline{m})}{n^s m^{1-s}}
		V_\mathfrak{a} \left( \frac{mn}{q} \right)
		\ll_\epsilon q^\epsilon (1+|t|)^\epsilon,
		\quad
		(s=1/2 + it),
	\end{equation}
	where
	\begin{equation} \label{eq:AFE2}
		V_\mathfrak{a} \left( y \right)
		=
		\frac{1}{2 \pi i} \int\limits_{(1)}
		\pi^{-w} 
		\frac{ e^{w^2} \cos^2(\pi w)}{w}
		\frac{\Gamma \left( \frac{\frac{1}{2} + w + \mathfrak{a}}{2} \right)^2}{\left( \Gamma \frac{\frac{1}{2}+\mathfrak{a}}{2} \right)^2}
		y^{-w} dw,
		\quad
		(\aa = 0,1),
	\end{equation}
	and $\int_{(b)}$ denotes $\int_{b - i \infty}^{b + i \infty}$.
\end{conj}

In Section \ref{sec:Heuristics}, we present heuristics in support of the above Lindel\"of hypothesis \eqref{eq:newbeef}. We apply a special case of the conditional estimate \eqref{eq:newbeef} to bound a logarithmic bilinear form. This is
\begin{prop} \label{prop:onLH}
	Assume Conjecture \ref{conj:LH} for prime moduli $p$. Let $\delta >0$. There exists an arbitrarily small $\epsilon >0$, such that, for $s = 1 +\epsilon + it$, we have, for any $1 \le N_1 < N$,
	\begin{align} \label{eq:750h}
		&\sum_{N_1 < n < N}
		\frac{\tau_2(n) \Kl_2(an, p)}{n^s}
		\ll_\epsilon
		p^{1/2+\epsilon} (N_1)^{-1/2},
	\end{align}
	where
	\begin{equation} \label{eq:Kl2def}
			\Kl_2(A, p)
			=\sum_{\substack{xy \equiv A \pmod p\\ x,y \in \F_p^\times}}
			e_p(x + y),
			\quad
			(A, p) = 1.
		\end{equation}
\end{prop}
\begin{rmk}
	The work \cite{KMS2020} (2020)	of Kowalski, Michel, Sawin on stratification of exponential sums seems appropriate for our \eqref{eq:750h}, were it to hold for shorter sums. More precisely, a special case of \cite[Corollary 1.4]{KMS2020} yields for any $\epsilon >0$, there exists $\delta = \delta(\epsilon) >0$ such that for any
	\begin{equation} \label{eq:short}
		N \ge p^{2/3 + \epsilon},
	\end{equation}
	we have (in our normalization)
	\begin{align} \label{eq:750i}
		&\sum_{n \le N}
		\tau_2(n) \Kl_2(an, p)
		\ll_\epsilon
		N p^{\frac{1}{2} -\delta}
	\end{align}
	for any $(a,p)=1$. The restriction \eqref{eq:short} is incompatible with partial summations to yield \eqref{eq:750h}. At the moment, we are unable to obtain Proposition \ref{prop:onLH} unconditionally, but we believe that the left side of \eqref{eq:750h} could be more manageable to treat than removing \eqref{eq:short} from that of \eqref{eq:750i}, due to presence of the logarithmic weight $1/n$ in the bilinear form \eqref{eq:750h}. 
\end{rmk}

Using the conditional Proposition \ref{prop:onLH}, which is the only conditional part of the paper, we are able to beat Selberg's two-thirds level and obtain level four-fifths for $\tau_2$ for prime moduli with smooth weights.

\begin{thm} \label{thm:beatsSelberg}
	Assume Conjecture \ref{conj:levelone} for prime moduli $p$. Let $W(x)$ be a smooth test weight supported on $[1,2]$ as in Conjecture \ref{conj:levelone}.
	Then, for any $\epsilon>0$, there exists $\delta = \delta(\epsilon) >0$ such that
	\begin{equation} \label{eq:tau2b}
		\sum_{\substack{n\\ n\equiv a \pmod p}}
		\tau_2(n) W\left(\frac{n}{X}\right)
		=
		\frac{1}{p-1}
		\sum_{\substack{n\\ (n,p) = 1}}
		\tau_2(n) W\left(\frac{n}{X}\right)
		+ O_\delta \left( \frac{X^{1-\delta}}{p} \right)
	\end{equation}
	uniformly for all $(a,p) = 1$ and $p \le X^{\theta_{W,2} - \epsilon}$, with
	\begin{equation} \label{eq:4over5}
		\theta_{W,2} = 4/5.
	\end{equation}
\end{thm}

\begin{rmk}
	By carefully constructing smooth weights that have their Mellin transforms vanishing with high orders at the integers (so that these zeros cancel with the poles from the gamma factors), one can shift the relevant contours far left (see the bound \eqref{eq:204}) and push \eqref{eq:4over5} to full level $\theta_{W ,2} = 1$. This improvement however does not seem to give substantially better results to primes than the level $4/5$ obtain here, so we will not pursue this direction in this paper.
\end{rmk}

Next, we will consider the various ways our approach can be applied.

\subsection{Applications} \label{sec:apps}

Firstly, our method here is likely to be effective to extend from prime to composite moduli; see, e.g, \cite{NguyenVar} where a similar extension was done. Once extended to at least square-free moduli (which are sufficient for applications to primes), a prospective use of the conditional level $4/5$ for $\tau_2$ from Theorem \ref{thm:beatsSelberg} is an improvement to Selberg's result \eqref{eq:SelbergTP} on approximations to twin primes, as alluded to earlier. Indeed, \eqref{eq:lambda} yields, for $\theta_2 = 4/5$, that
\begin{equation} \label{eq:selbergb}
	\lambda_2 = 12,
\end{equation}
which, in turns, eliminates the possibility that $\nu(n) = 3$ and $\nu(n+2) = 2$--see \eqref{eq:1237} above--and implies that at least one of the following 5 sets has positive asymptotic density, in particular, is infinite:
\begin{align} \label{eq:foursets}
	S_1 &= \{ n \in \N\ |\ \text{$n$ and $n+2$ are both primes} \},
	\\
	S_2 &= \{ n \in \N\ |\  \text{$n$ is prime and $n+2$ has at most two prime factors} \},
	\\
	S_3 &= \{ n \in \N\ |\  \text{$n+2$ is prime and $n$ has at most two prime factors} \},
	\\
	S_4 &= \{ n \in \N\ |\  \text{$n$ and $n+2$ each has at most two prime factors} \},
	\\
	S_5 &= \{ n \in \N\ |\  \text{$n+2$ has at most two and $n$ has at most three prime factors} \}.
\end{align}
Following this, we will investigate the broader implications of our approach.

In certain applications, proving infinitely many, rather than a positive proportion, is enough. One may ask whether \eqref{eq:selberg} could be generalized to an affine inequality of the form
\begin{equation} \label{eq:ktuple}
	\sum_{h \in \mathcal{H}}
	c_h
	k^{\nu(n + h)}
	< \frac{C_{\mathcal{H}}(k)}{\theta_k}
	+ D_{\mathcal{H}}(k) + \epsilon,
	\quad
	(\forall \epsilon>0, k, \mathcal{H}),
\end{equation}
which holds for \textbf{infinitely many} $n$, possibly with a modified higher dimensional $\Lambda^2$-sieve in conjunction with pioneering and innovative sieving ideas of Y.-T. Zhang \cite{Zhang2022} (2022), for some explicit enough quantities $C_{\mathcal{H}}(k)$ and $D_{\mathcal{H}}(k)$ depending on $k \in \N$, $\mathcal{H}$ a finite set of distinct even natural numbers, $c_h$ positive real numbers depending only on $h\in \mathcal{H}$, and, of course, on the level of distribution $\theta_k \in (0,1)$ (whose definitions are analogous to that for $\theta_2$) for the $k$-fold divisor function $\tau_k(n)$. The basic observation is that the left side of \eqref{eq:ktuple} is large when $k$ grows. Of special interests are, for a given even natural number $h$, less than 246, say, whether there exists a natural number $k = k(h)$ and two positive real numbers $c_0, c_h$ all depending on $h$, such that
\begin{equation} \label{eq:TTP}
	c_0 k + c_h k^2
	>
	\frac{C_{\{ 0, h \}}(k)}{\theta_k}
	+ D_{\{ 0, h \}}(k),
	\quad
	(\text{for sufficiently large $k$}),
\end{equation}
as \eqref{eq:TTP} would imply that there are infinitely many pairs of primes $\{n, n+h \}$ whose gap is equal to $h$, a weak form of a special case (\cite[Conjecture B, p. 42]{HL1923}) of the Hardy-Littlewood $k$-tuple Conjectures (note that this $k$ in $k$-tuple, which is the size of the set $\mathcal{H}$, is different from the $k$ in \eqref{eq:TTP}). We are optimistic that our approach \eqref{eq:TTP} to bounded gaps between primes is tractable. Let us briefly elaborate.

From Deligne's spectacular resolutions \cite{DeligneWeil1, DeligneWeil2} of the Weil Conjectures for varieties over finite fields, it is known that the level $\theta_k$ is at least $2/(k+1)$ for all $k$; see, e.g, \cite[Table 1, p. 33]{Nguyen2021}. However, for a special set of moduli, the author obtained in \cite[Theorem 1, p. 35]{Nguyen2021} that, on average, the level $\theta_k$ is uniformly bounded below by an absolute constant larger than 1/2, independent of $k$. Thus, we save a whole factor of $k$ from the first term on the right side of \eqref{eq:TTP}. One down, two more to go. We note that for sieves, an average level of distribution, rather than individual levels, is all that needed--think Bombieri-Vinogradov Theorem. In addition, it should be mentioned that our approach/proposal here towards bounded gaps between primes is conceptually different than prior approach of Goldston-Pintz-Yıldırım \cite{GPYI, GPYII} (2009), made work in sensational fashion by Zhang \cite{Zhang2014} (2014). Ours is an infinite process in which a limit is involved, whereas G.P.Y. did it all in one fell swoop, requiring only a single average level of distribution from $\Lambda(n)$.

Lastly, it is worth pondering whether the result in very recent work \cite{Sharma23} just out from Sharma on $\theta_3$, which now holds for all square-free moduli, based on different method via Munshi's separation of oscillation technique and asserts that $\theta_3 = 1/2 + 1/30$, which is already forefront by itself, packs enough juice to eliminate one or more of the five sets in \eqref{eq:foursets}. We are eager to keep exploring these ideas in a near future.

\section{Lemmas}

The following two results are well-known, but requires some rearranging, and we include the proofs here for completeness.

\begin{lemma}[F.E. for zeta]
	We have
	\begin{equation} \label{eq:FEzeta}
		\zeta(1-s)
		=
		2^{1-s} \pi^{-s}
		\sin\left(\frac{1}{2} (1-s) \pi \right)
		\Gamma(s)
		\zeta(s).
	\end{equation}
\end{lemma}
\begin{proof}
	By equation (2.1.9) on page 16 of Titchmarsh, we have
	\begin{equation} \label{eq:604}
		\zeta(s) = \chi(s) \zeta(1-s),
	\end{equation}
	where
	\begin{equation}
		\chi(s)
		=
		2^s \pi^{s-1} \sin \frac{1}{2} \pi s \Gamma(1-s).
	\end{equation}
	Substituting in $s-1$ for $s$ in \eqref{eq:604} gives \eqref{eq:FEzeta}.
\end{proof}

\begin{lemma}[FE for $L(s,\chi)$]
	For $\chi_1 \pmod d$ a primitive character, we have
	\begin{equation} \label{eq:FELschi}
		L(1-s, \overline{\chi_1})
		=
		i^\aa \pi^{\frac{1}{2} - s}
		d^{s-1}
		\frac{\Gamma\left(\frac{1}{2}(s+\aa)\right)}{\Gamma\left(\frac{1}{2} (1-s+\aa)\right)}
		\overline{\tau(\chi_1)}
		L(s,\chi_1),
	\end{equation}
	where
	\begin{equation} \label{eq:FE}
		\aa
		= \begin{cases}
			0, & \text{if } \chi_1(-1)=1,
			\\
			1, & \text{if } \chi_1(-1)= -1.
		\end{cases}
	\end{equation}
\end{lemma}
\begin{proof}
	By equation (9.14) on page 71 of Davenport, we have
	\begin{equation}
		\xi(1-s, \overline{\chi}_1)
		=
		\frac{i^\aa d^{1/2}}{\tau(\chi_1)} \xi(s, \chi_1),
	\end{equation}
	where
	\begin{equation}
		\xi(s, \chi_1)
		=
		(\pi /d)^{-\frac{1}{2}(s+\aa)}
		\Gamma\left[\frac{1}{2}(s+\aa)\right] L(s,\chi_1).
	\end{equation}
	The equation \eqref{eq:FELschi} then follows by rearranging and using the relation
	\begin{equation}
		\frac{d^{1/2}}{\tau(\chi_1)}
		=
		\frac{\overline{\tau(\chi_1)}}{d^{1/2}}.
	\end{equation}
\end{proof}

In applying the functional equation, the arguments of the gamma factors take slightly different shape depending if the character is odd or even. We need the following lemma to sum over the different cases.
\begin{lemma}
	We have
	\begin{equation} \label{eq:orthogonality}
		\Star\sum_{\chi_1 \pmod p}
		\chi_1(A)
		\overline{\chi_1}(B)
		=
		\begin{cases}
			p-2, & \text{ if } A \equiv B \pmod p,
			\\
			-1, & \text{ otherwise},
		\end{cases}
	\end{equation}
	\begin{equation} \label{eq:orthogonalityeven}
		\Flat\sum_{\chi_1 \pmod p}
		\chi_1(A)
		\overline{\chi_1}(B)
		=
		\frac{1}{2}
		\begin{cases}
			p-2, & \text{ if } A \equiv \pm B \pmod p,
			\\
			-2, & \text{ otherwise},
		\end{cases}
	\end{equation}
	\begin{equation} \label{eq:orthogonalityodd}
		\Sharp\sum_{\chi_1 \pmod p}
		\chi_1(A)
		\overline{\chi_1}(B)
		=
		\frac{1}{2}
		\begin{cases}
			p-1, & \text{ if } A \equiv B \pmod p,
			\\
			1 - p, & \text{ if } A \equiv - B \pmod p,
			\\
			0, & \text{ otherwise}.
		\end{cases}
	\end{equation}
	where the $\flat$ and $\sharp$ denote summing over even and odd primitive characters, respectively. Here, $A \equiv \pm B \pmod p$ means that either $A \equiv B \pmod p$ or $A \equiv - B \pmod p$.
\end{lemma}
\begin{proof}
	The relation \eqref{eq:orthogonality} follows from summing over all characters then subtracting off contribution from the principal character. 
	
	Next, we have
	\begin{equation}
		\frac{1+ \chi(-1)}{2}
		=
		\begin{cases}
			1, & \text{ if $\chi(-1) = 1$ is even,}
			\\
			0, & \text{ if $\chi(-1) = -1$ is odd.}
		\end{cases}
	\end{equation}
	With the above, the left side of \eqref{eq:orthogonalityeven} is equal to
	\begin{align}
		\Star\sum_{\chi_1 \pmod p}
		\chi_1(A)
		\overline{\chi_1}(B)
		\left(\frac{1+ \chi(-1)}{2}\right)
		&=
		\frac{1}{2}
		\left(
		\begin{cases}
			\varphi(p), & \text{ if } A \equiv B \pmod p
			\\
			0, & \text{ otherwise}
		\end{cases}
		- 1
		\right)
		\\&\quad +
		\frac{1}{2}
		\left(
		\begin{cases}
			\varphi(p), & \text{ if } A \equiv - B \pmod p
			\\
			0, & \text{ otherwise}
		\end{cases}
		- 1.
		\right)
	\end{align}
	This leads to the right side of \eqref{eq:orthogonalityeven} after simplifications.
	
	Similarly, by detecting odd primitive characters by
	\begin{equation}
		\frac{1 - \chi(-1)}{2},
	\end{equation}
	we obtain the relation \eqref{eq:orthogonalityodd}.
\end{proof}

Summing over the Gauss sums coming from the functional equation give a hyper-Kloosterman sum.

\begin{lemma}
	We have
	\begin{equation} \label{eq:genKloosa}
			\displaystyle\sideset{}{^\flat}\sum_{\chi_1 \pmod p}
			\chi_1(an)
			\overline{\tau(\chi_1)}^k
			=
			\frac{1}{2}
			(p-2)
			{\Kl_k(an, p)}
			+ (-1)^{k + 1}
		\end{equation}
	and
	\begin{equation} \label{eq:genKloosOdd}
		\displaystyle\sideset{}{^\sharp}\sum_{\chi_1 \pmod p}
		\chi_1(an)
		\overline{\tau(\chi_1)}^k
		=
		\frac{1}{2}
		(p-1)
		\left(
		{\Kl_k(an, p)}
		+
		{\Kl_k(-an, p)}
		\right)
		+ (-1)^{k},
	\end{equation}
	where
	\begin{equation} \label{eq:Klkdef}
		\Kl_k(A, p)
		=\sum_{\substack{x_1 \cdots x_k \equiv A \pmod p\\ x_1, \dots, x_k \in \F_p^\times}}
		e_p(x_1 + \dots + x_k),
		\quad
		(A, p) = 1.
	\end{equation}
\end{lemma}
\begin{proof}
	By \eqref{eq:Klkdef} and \eqref{eq:orthogonalityeven}, we have
	\begin{align}
			\displaystyle\sideset{}{^\flat}\sum_{\chi_1 \pmod p}
			\chi_1(na)
			\overline{\tau(\chi_1)}^k
			&=
			\Star\sum_{a_1, \dots, a_k \pmod p}
			e_p(-a_1 - \dots - a_k)
			\Flat\sum_{\chi_1 \pmod p}
			\chi_1(na)
			\overline{\chi_1}(a_1  \dots a_k)
			\\&
			\stackrel{\eqref{eq:orthogonalityeven}}{=}
			\Star\sum_{a_1, \dots, a_k \pmod p}
			e_p(-a_1 - \dots - a_k)
			\frac{1}{2}
			\begin{cases}
					p -2, & \text{ if } a_1 \dots a_k \equiv \pm an \pmod p
					\\
					-2, & \text{ otherwise }
				\end{cases}
			\\
			&=
			\frac{1}{2} (p-2)
			\Star\sum_{\substack{a_1, \dots, a_k \pmod p\\ a_1 \dots a_k \equiv \pm an \pmod p}}
			e_p(-a_1 - \dots - a_k)
			-
			\Star\sum_{a_1, \dots, a_k \pmod p}
			e_p(-a_1 - \dots - a_k)
			\\
			&= \frac{1}{2} (p-2) \Kl_k( \pm an, p)
			- \mu(p)^k.
		\end{align}
	This leads to the right side of \eqref{eq:genKloosa}.
\end{proof}

\begin{lemma} \label{lem:13}
	We have, for $(a,p)=1$ and $a>0$,
	\begin{equation} \label{eq:801}
		\Flat\sum_{\chi_1 \pmod p}
		\chi_1(n)
		\overline{\chi_1}(am)
		\tau(\chi_1)
		=
		\frac{p-4}{2} e_p(a m \overline{n}) - 1
	\end{equation}
	and
	\begin{equation} \label{eq:801c}
		\Sharp\sum_{\chi_1 \pmod p}
		\chi_1(n)
		\overline{\chi_1}(am)
		\tau(\chi_1)
		=
		\frac{p-1}{2} e_p(a m \overline{n}).
	\end{equation}
\end{lemma}
\begin{proof}
	By definition of the Gauss sum, we have
	\begin{align}
		\Flat\sum_{\chi_1 \pmod p}
		\chi_1(n)
		\overline{\chi_1}(am)
		\tau(\chi_1)
		&=
		\Flat\sum_{\chi_1 \pmod p}
		\chi_1(n)
		\overline{\chi_1}(am)
		\Star\sum_{a_1 \pmod p}
		e_p(a_1) 
		\chi_1(a_1)
		\\
		&=
		\Star\sum_{a_1 \pmod p}
		e_p(a_1) 
		\Flat\sum_{\chi_1 \pmod p}
		\chi_1(a_1 n)
		\overline{\chi_1}(am).
	\end{align}
	By \eqref{eq:orthogonalityeven}, the above is
	\begin{align}
		&=
		\frac{p-2}{2}
		\Star\sum_{\substack{a_1 \pmod p\\ a_1 \equiv \pm a m \overline{n} \pmod p}}
		e_p(a_1) 
		-
		\Star\sum_{\substack{a_1 \pmod p\\ a_1 \not\equiv \pm a m \overline{n} \pmod p}}
		e_p(a_1) 
		\\&=
		\frac{p-2}{2}
		e_p( \pm a m \overline{n})
		-
		\left(
		\Star\sum_{\substack{a_1 \pmod p}}
		e_p(a_1) 
		-
		\Star\sum_{\substack{a_1 \pmod p\\ a_1 \equiv \pm a m \overline{n} \pmod p}}
		e_p(a_1) 
		\right)
		\\&=
		\frac{p-2}{2}
		e_p( \pm a m \overline{n})
		-
		\left(
		\mu(p)
		-
		e_p( \pm a m \overline{n})
		\right).
	\end{align}
	This gives the right side of \eqref{eq:801} after simplification. Similarly, by \eqref{eq:orthogonalityodd}, we obtain \eqref{eq:801c}.
\end{proof}

\begin{lemma} \label{lem:Estermann}
	Let $q \ge 1$ and $(A,q)=1$. Define
	\begin{equation} \label{eq:Ddef}
		E_2(s; q, A)
		=
		\sum_{n=1}^{\infty}
		\frac{\tau_2(n) e_q(n A)}{n^s},
		\quad (\sigma >1).
	\end{equation}
	Then, the function $E_2(s; q, A)$ has a meromorphic continuation to all of $\C$ with a pole of order two at $s=1$ and satisfies the functional equation
	\begin{equation} \label{eq:DFE}
		E_2(1-s; q, A)
		=
		2 G^2(1-s) q^{2s - 1}
		\left[
		\cos(\pi (1-s))
		E_2(s;q,  - \overline{A})
		-
		E_2(s; q, \overline{A})
		\right],
	\end{equation}
	where
	\begin{equation} \label{eq:Gs}
		G(s)
		=
		- i (2\pi)^{s-1} \Gamma(1-s).
	\end{equation}
	Moreover, $E_2(s; q, A)$ has the same polar part as $q^{1-2s} \zeta^2(s)$, i.e.,
	\begin{equation}
		E_2(s; q, A)
		=
		\frac{1}{q}
		\left[  
		\frac{1}{(s-1)^2}
		+ \frac{2(\gamma - \log q)}{s-1}
		+ \dots.
		\right],
		\quad (s \to 1).
	\end{equation}
	In particular, the residue of $E_2(s; q, A)$ at $s=1$ is independent of $A$ and  equal to
	\begin{equation} \label{eq:resofE2}
		2 q^{-1} (\gamma - \log q).
	\end{equation}
\end{lemma}
\begin{proof}
	This classic result is due to Estermann
	\cite{Estermann1930}.
\end{proof}

Using the above Lemma, we obtain

\begin{lemma} \label{lem:D2ACFE}
	Let $q \ge 1$ and $(A,q)=1$. Define
	\begin{equation} \label{eq:D2def}
		D_2(s; q, A)
		=
		\sum_{n=1}^{\infty}
		\frac{\tau_2(n) {\Kl_2(n A; q)}}{n^s},
		\quad (\sigma >  1).
	\end{equation}
	Then, the function $D_2(s; q , A)$ has a meromorphic continuation to all of $\C$ with a double pole at $s=1$. It has the same polar part as $\mu(q) q^{1-2s} \zeta^2(s)$, i.e.,
	\begin{equation}
		D_2(s; q, A)
		=
		\frac{\mu(q)}{q}
		\left[  
		\frac{1}{(s-1)^2}
		+ \frac{2(\gamma - \log q)}{s-1}
		+ \dots.
		\right],
		\quad (s \to 1).
	\end{equation}
	In particular, the residue of $D_2(s; q, A)$ at $s=1$ is independent of $A$ and  equal to
	\begin{equation} \label{eq:resofD2}
		2 \mu(q) q^{-1} (\gamma - \log q).
	\end{equation}
	Moreover, the function $D_2(s; q ,A)$ satisfies the ``functional equation", interchanging $s \leftrightarrow 1-s$:
	\begin{equation} \label{eq:D2FE}
		D_2(s; q ,A)
		=
		2 G^2(s) q^{1-2s}
		\sum_{q = dr} d \mu(r)
		\left(
		\cos(\pi s)
		\sum_{\substack{n=1\\ n \equiv A \pmod d}}^\infty
		\frac{\tau_2(n)}{n^{1-s}}
		-
		\sum_{\substack{n=1\\ n \equiv - A \pmod d}}^\infty
		\frac{\tau_2(n)}{n^{1-s}}
		\right),
	\end{equation}
	with $G(s)$ given as in \eqref{eq:Gs}.
\end{lemma}
\begin{proof}
	By \eqref{eq:Kl2def}, we can write the right side of \eqref{eq:D2def} as
	\begin{align}
		\sum_{n=1}^{\infty}
		\frac{\tau_2(n)}{n^s}
		\Star\sum_{a \pmod q}
		e_q (- a - \overline{a} nA)
		&=
		\Star\sum_{a \pmod q}
		e_q (- a)
		\sum_{n=1}^{\infty}
		\frac{\tau_2(n) e_q(- n \overline{a} A )}{n^s}
		\\&=
		\Star\sum_{a \pmod q}
		e_q (- a)
		E_2(s; q, - \overline{a} A).
	\end{align}
	This expression gives the meromorphic continuation to $D_2(s; q, A)$ to all of $\C$ with a double pole at $s=1$, via that of $E_2(s; q, - \overline{a} A)$ from Lemma \ref{lem:Estermann} and also the claim about the residue.
	
	Next, by the functional equation \eqref{eq:DFE} for $E_2(s; - \overline{a} A, q)$, the above is equal to
	\begin{equation} \label{eq:544}
		D_2(s;q, A)
		=
		2 G^2(s) q^{1-2s}
		\left[
		\cos(\pi s)
		\Star\sum_{a \pmod q}
		e_q (- a)
		E_2(1-s; q,  a \overline{A})
		-
		\Star\sum_{a \pmod q}
		e_q (a)
		E_2(1-s; q, - a \overline{A})
		\right].
	\end{equation}
	By the definition \eqref{eq:Ddef}, we have, for $\Re(1-s) > 1$,
	\begin{equation} \label{eq:541}
		\Star\sum_{a \pmod q}
		e_q (- a)
		E_2(1-s; q, a \overline{A})
		=
		\sum_{n=1}^{\infty}
		\frac{\tau_2(n)}{n^{1-s}}
		\Star\sum_{a \pmod q}
		\overline{
			e_q (a - a \overline{A} n) }
		=
		\sum_{n=1}^{\infty}
		\frac{\tau_2(n)}{n^{1-s}}
		c_q(1-\overline{A} n),
	\end{equation}
	where $c_q(a)$ is Ramanujan's sum. By the property
	\begin{equation}
		c_q(1-\overline{A} n)
		=
		\sum_{\substack{q = dr\\ d \mid 1 - \overline{A} n}}
		d \mu(r),
	\end{equation}
	the expression \eqref{eq:541} is equal to
	\begin{equation}
		\Star\sum_{a \pmod q}
		e_q (- a)
		E_2(1-s; q, a \overline{A})
		=
		\sum_{q = dr}
		d \mu(r)
		\sum_{\substack{n=1\\ n \equiv A \pmod q}}^{\infty}
		\frac{\tau_2(n)}{n^{1-s}},
		\quad (\Re(1-s) > 1).
	\end{equation}
	Similarly, we have
	\begin{equation} \label{eq:541b}
		\Star\sum_{a \pmod q}
		e_q (a)
		E_2(1-s; q, - a \overline{A})
		=
		\sum_{q = dr}
		d \mu(r)
		\sum_{\substack{n=1\\ n \equiv - A \pmod q}}^{\infty}
		\frac{\tau_2(n)}{n^{1-s}},
		\quad (\Re(1-s) > 1).
	\end{equation}
	Thus, by the above two, \eqref{eq:544} is equal to
	\begin{equation}
		2 G^2(s) q^{1-2s}
		\left[
		\cos(\pi s)
		\sum_{q = dr}
		d \mu(r)
		\sum_{\substack{n=1\\ n \equiv A \pmod q}}^{\infty}
		\frac{\tau_2(n)}{n^{1-s}}
		-
		\sum_{q = dr}
		d \mu(r)
		\sum_{\substack{n=1\\ n \equiv - A \pmod q}}^{\infty}
		\frac{\tau_2(n)}{n^{1-s}}
		\right],
	\end{equation}
	which is the right side of \eqref{eq:D2FE}.
\end{proof}

In particular, for $q = p$ is prime, \eqref{eq:D2FE} reduces to
\begin{lemma}
	Let $D_2(s; q, A)$ be defined as in \eqref{eq:D2def}. We have
	\begin{align} \label{eq:D2forqprime}
		\quad\quad 
		D_2(s; p, A)
		=
		2 G^2(s) p^{1 - 2s}
		&\left(
		p \left( \cos(\pi(s)) \sum_{n \equiv A \pmod p}  \frac{\tau_2(n)}{n^{1-s}} - \sum_{n \equiv - A \pmod p} \frac{\tau_2(n)}{n^{1-s}} \right)
		\right. 
		\\ &\left. \quad
		- \left( \cos(\pi(s)) - 1 \right) \zeta^2(1-s)
		\right).
	\end{align}
\end{lemma}

\begin{lemma}
	We have the following convexity bound for $D_2(s; q, A)$ on the critical strip:
	\begin{equation} \label{eq:D2convexbd}
		|D_2(s; q, A)|
		\ll_{\epsilon, B}
		q^{-3 \sigma/2 + 2 + \epsilon}
		(1 + |t|)^{- B(1 - \sigma)},
		\quad
		(0 \le \sigma \le 1),
	\end{equation}
	for any large but fixed $B>0$, uniformly in $A$. In particular, we have, on the critical line,
	\begin{equation} \label{eq:D2convexbdoncriticalline}
		|D_2(1/2 + it; q, A)|
		\ll_{\epsilon, B}
		q^{5/4 + \epsilon}
		(1 + |t|)^{- B},
	\end{equation}
	for any large but fixed $B>0$, uniformly in $A$.
\end{lemma}
\begin{proof}
	At $s = 1 + \epsilon + it$, we have, by \eqref{eq:D2def} and Weil's bound,
	\begin{equation}
		|D_2(1+\epsilon + it)|
		\ll_\epsilon
		q^{1/2 + \epsilon}.
	\end{equation}
	At $s = -\epsilon + it$, we have, by \eqref{eq:D2FE},
	\begin{equation}
		|D_2(- \epsilon + it)|
		\ll_\epsilon
		|\Gamma(1+\epsilon)|^2
		q^{1+ 2 \epsilon}
		q^{1+\epsilon}
		A^{-1-\epsilon}
		\ll_{\epsilon, B}
		q^{2+\epsilon}
		(1 + |t|)^{-B},
	\end{equation}
	for any large but fixed $B>0$. Hence, by convexity principle, we obtain
	\begin{equation}
		|D_2(s; q, A)|
		\ll_{\epsilon, B}
		q^{-(2-1/2) \sigma + 2 + \epsilon}
		(1 + |t|)^{- B(1 - \sigma)},
		\quad
		( 0 \le \sigma \le 1),
	\end{equation}
	for any large but fixed $B>0$.
\end{proof}

\begin{lemma}[AFE] \label{lem:2}
	For $s=1/2+it$, we have, for primitive $\chi_1 \pmod p$,
	\begin{equation} \label{eq:AFE}
		L(1-s,\chi_1)
		L(s, \overline{\chi_1})
		=
		2 \sum_{m,n}
		\frac{\chi(m) \overline{\chi}(n)}{m^{1-s} n^{s}}
		V_\mathfrak{a} \left( \frac{mn}{p} \right),
	\end{equation}
	with $V_\mathfrak{a}(\cdot)$ given in \eqref{eq:AFE2}.
\end{lemma}
\begin{proof}
	This is a straight forward modification of \cite[Lemma 2, p. 241]{Soundararajan2007}.
\end{proof}

\begin{lemma}
	Let $D_2(s; p, A)$ be defined as in \eqref{eq:D2def}. We have, for $\sigma >1$,
	\begin{align} \label{eq:739b}
		\quad
		&D_2(s; p, A) = 
		\frac{1}{p-1} \zeta^2(s)
		\\& \quad +
		\frac{p}{p-1} 
		\Star\sum_{\chi_1 \pmod p}
		\overline{\chi_1}(A)
		\tau(\chi_1)
		i^\aa \pi^{s - \frac{1}{2}}
		p^{s-1}
		\frac{\Gamma\left(\frac{1}{2}(1-s+\aa)\right)}{\Gamma\left(\frac{1}{2} (s+\aa)\right)}
		L(1-s,\chi_1)
		L(s, \overline{\chi_1}).
	\end{align}
	The right side of the above is valid for all of $\C$ except for $s=1$, thus, by principal of analytic continuation, \eqref{eq:739b} is also valid for all of $\C$ except for $s=1$.
\end{lemma}
\begin{proof}
	By definition of $Kl_2$, we have, for $\sigma >1$,
	\begin{equation}
		D_2(s; p, A) = 
		\Star\sum_{c \pmod p} e_p(c)
		\sum_{n=1}^\infty
		\frac{\tau_2(n) e_p( n A \overline{c})}{n^s}.
	\end{equation}
	Going to Dirichlet characters by
	\begin{equation}
		e_p( n A \overline{c})
		=
		\frac{1}{p-1} 
		\sum_{\chi \pmod p}
		\overline{\chi}(n A \overline{c})
		\tau(\chi)
	\end{equation}
	we get
	\begin{align} \label{eq:739}
		D_2(s; p, A) &= 
		\Star\sum_{c \pmod p} e_p(c)
		\frac{1}{p-1} 
		\sum_{\chi \pmod p}
		\overline{\chi}(A \overline{c})
		\tau(\chi)
		\sum_{n=1}^\infty
		\frac{\tau_2(n) \overline{\chi}(n)}{n^s}
		\\&=
		\frac{1}{p-1} \zeta^2(s)
		+
		\frac{1}{p-1} 
		\Star\sum_{\chi_1 \pmod p}
		\overline{\chi_1}(A)
		\tau^2(\chi_1)
		L^2(s, \overline{\chi_1}),
	\end{align}
	where we have split up the principal character, then used $\tau(\chi_0) = \mu(p)=-1$, and $\Star\sum_{c \pmod p} e_p(c) \overline{\chi_1}(\overline{c}) = \tau(\chi_1)$ to simplify. We next apply the functional equation to just one of the $L$-functions.
	
	By \eqref{eq:FELschi}, we have
	\begin{equation}
		L(s, \overline{\chi_1})
		=
		i^\aa \pi^{s - \frac{1}{2}}
		p^{s-1}
		\frac{\Gamma\left(\frac{1}{2}(1-s+\aa)\right)}{\Gamma\left(\frac{1}{2} (s+\aa)\right)}
		\overline{\tau(\chi_1)}
		L(1-s,\chi_1),
	\end{equation}
	thus the second factor in \eqref{eq:739} becomes
	\begin{align}
		&\frac{1}{p-1} 
		\Star\sum_{\chi_1 \pmod p}
		\overline{\chi_1}(A)
		\tau^2(\chi)
		L^2(s, \overline{\chi_1})
		\\ & \quad =
		\frac{1}{p-1} 
		\Star\sum_{\chi_1 \pmod p}
		\overline{\chi_1}(A)
		\tau^2(\chi_1)
		i^\aa \pi^{s - \frac{1}{2}}
		p^{s-1}
		\frac{\Gamma\left(\frac{1}{2}(1-s+\aa)\right)}{\Gamma\left(\frac{1}{2} (s+\aa)\right)}
		\overline{\tau(\chi_1)}
		L(1-s,\chi_1)
		L(s, \overline{\chi_1})
		\\& \quad =
		\frac{p}{p-1} 
		\Star\sum_{\chi_1 \pmod p}
		\overline{\chi_1}(A)
		\tau(\chi_1)
		i^\aa \pi^{s - \frac{1}{2}}
		p^{s-1}
		\frac{\Gamma\left(\frac{1}{2}(1-s+\aa)\right)}{\Gamma\left(\frac{1}{2} (s+\aa)\right)}
		L(1-s,\chi_1)
		L(s, \overline{\chi_1}).
	\end{align}
	Combining the above with \eqref{eq:739}, we obtain \eqref{eq:739b}.
\end{proof}

\begin{lemma} \label{lemma:Mellin}
	By inverse Mellin transform, we have
	\begin{equation} \label{eq:Mellin}
		w\left(\frac{n}{X}\right)
		=
		\frac{1}{2\pi i}
		\int_{(1+\epsilon)}
		\mathcal{M}[w](s)
		\left(\frac{n}{X}\right)^{-s} ds.
	\end{equation}
\end{lemma}

\section{Proof of Theorem \ref{thm:1}}
Writing the congruence condition $n \equiv a \pmod p$ as a sum over multiplicative characters
\begin{equation} \label{eq:delta}
	\frac{1}{p-1}
	\sum_{\chi \pmod p}
	\chi(a) \overline{\chi}(n)
	=
	\begin{cases}
		1, & \text{if } n \equiv a \pmod p,
		\\
		0,& \text{otherwise},
	\end{cases}
\end{equation}
we have
\begin{equation} \label{eq:535}
	\sum_{\substack{n\\ n\equiv a \pmod p}} 
	\tau_2(n) w\left(\frac{n}{X}\right)
	=
	\frac{1}{p-1}\
	\sum_{\substack{n\\ (n,p)=1}} 
	\tau_2(n) w\left(\frac{n}{X}\right)
	+
	\frac{1}{p-1}\
	\Prime\sum_{\chi \pmod p}
	\chi(a)
	\sum_{n}
	\tau_2(n) \overline{\chi}(n) w\left(\frac{n}{X}\right),
\end{equation}
where we have split up contribution from the principal character. Denote
\begin{equation}
	\Delta_{2, w}(X; p ,a)
	=
	\frac{1}{p-1}\
	\Prime\sum_{\chi \pmod p}
	\chi(a)
	\sum_{n}
	\tau_2(n) \overline{\chi}(n) w\left(\frac{n}{X}\right)
\end{equation}
the second term on the right side of \eqref{eq:535}. Since $p$ is prime, all non-principal characters modulo $p$ are primitive. Thus, we have
\begin{equation}
	\Delta_{2, w}(X; p ,a)
	=
	\frac{1}{p-1}\
	\Star\sum_{\chi_1 \pmod p}
	\chi_1(a)
	\sum_{n}
	\tau_2(n) \overline{\chi}_1(n) w\left(\frac{n}{X}\right).
\end{equation}
By Mellin inversion \eqref{eq:Mellin}, this is equal to
\begin{equation}
	\frac{1}{p-1}\
	\Star\sum_{\chi_1 \pmod p}
	\chi_1(a)
	\frac{1}{2\pi i}
	\int\displaylimits_{(1+\epsilon)}
	\mathcal{M}[w](s) X^s
	L^2(s, \overline{\chi}_1)
	ds.
\end{equation}
The integrand in the above is entire. Shifting the line of integration left to $(-\epsilon)$ and making a change of variable $s \mapsto 1-s$, the above is equal to
\begin{equation}
	\frac{1}{p-1}\
	\Star\sum_{\chi_1 \pmod p}
	\chi_1(a)
	\frac{1}{2\pi i}
	\int\displaylimits_{(1+\epsilon)}
	\mathcal{M}[w](1-s) X^{1-s}
	L^2(1-s, \overline{\chi}_1)
	ds.
\end{equation}
By the functional equation \eqref{eq:FE}, this is
\begin{align} \label{eq:621}
	\frac{1}{p-1}\
	\Star\sum_{\chi_1 \pmod p}
	\chi_1(a)
	\frac{1}{2\pi i}
	\int\displaylimits_{(1+\epsilon)}
	\mathcal{M}[w](1-s) X^{1-s}
	\left(
	i^\aa \pi^{\frac{1}{2} - s}
	p^{s-1}
	\frac{\Gamma\left(\frac{1}{2}(s+\aa)\right)}{\Gamma\left(\frac{1}{2} (1-s+\aa)\right)}
	\overline{\tau(\chi_1)}
	L(s,\chi_1)\right)^2
	ds.
\end{align}
For $\sigma > 1 $, we have
\begin{equation}
	L^2(s,\chi_1)
	=
	\sum_{n=1}^\infty
	\frac{\tau_2(n) \chi_1(n)}{n^s}.
\end{equation}
We now bound contribution from the tail $n > p$. Push the line of integration right to $\sigma = B+1$, where $B>0$ a large but fixed constant. We have, for $s = B+1 + it$,
\begin{align} \label{eq:620}
	&\frac{1}{p-1}\
	\Star\sum_{\chi_1 \pmod p}
	\chi_1(a)
	\frac{1}{2\pi i}
	\int\displaylimits_{(B+\epsilon)}
	\mathcal{M}[w](1-s) X^{1-s}
	\left(
	i^\aa \pi^{\frac{1}{2} - s}
	p^{s-1}
	\frac{\Gamma\left(\frac{1}{2}(s+\aa)\right)}{\Gamma\left(\frac{1}{2} (1-s+\aa)\right)}
	\overline{\tau(\chi_1)}\right)^2
	\sum_{n > p}
	\frac{\tau_2(n) \chi_1(n)}{n^s}
	ds
	\\& \quad
	\ll
	\frac{1}{p-1} (p-1)
	X^{-B}
	p^{2(B + 1/2)}
	p^{-B + \epsilon}
	\ll
	\left(\frac{p}{X}\right)^B p
	\ll X^{-A}
\end{align}
for $p \le X^{1-\epsilon}$.
Thus, by \eqref{eq:620} and \eqref{eq:621}, we have, after some rearranging,
\begin{align} \label{eq:916}
	\Delta_{2, w}(X; p ,a)
	=
	C_{2,w}(X; p ,a)
	+ 
	L_{2,w}(X; p ,a)
	+
	O(X^{-A}),	
\end{align}
where
\begin{align} \label{eq:chan}
	C_{2,w}(X; p ,a)
	& =
	\frac{1}{p-1}
	X
	\sum_{\substack{n < p}}
	\tau_2(n)
	\Flat\sum_{\chi_1 \pmod p}
	\chi_1(an)
	\overline{\tau(\chi_1)}^2
	\\ &\quad  \times
	\frac{1}{2\pi i}
	\int\displaylimits_{(1+\epsilon)}
	(nX)^{-s}
	\mathcal{M}[w](1-s) 
	\left(
	\pi^{\frac{1}{2} - s}
	p^{s-1}
	\frac{\Gamma\left(\frac{s}{2}\right)}{\Gamma\left(\frac{1-s}{2} \right)}
	\right)^2
	ds
\end{align}
and
\begin{align} \label{eq:le}
	L_{2,w}(X; p ,a)
	&=
	\frac{1}{p-1}
	X
	\sum_{\substack{n < p}}
	\tau_2(n)
	\Sharp\sum_{\chi_1 \pmod p}
	\chi_1(an)
	\overline{\tau(\chi_1)}^2
	\\& \quad \times
	\frac{1}{2\pi i}
	\int\displaylimits_{(1+\epsilon)}
	(nX)^{-s}
	\mathcal{M}[w](1-s) 
	\left(
	i
	\pi^{\frac{1}{2} - s}
	p^{s-1}
	\frac{\Gamma\left(\frac{s+1}{2}\right)}{\Gamma\left(\frac{2-s}{2} \right)}
	\right)^2
	ds.
\end{align}
The treatments of $C_{2,w}(X; p ,a)$ and $L_{2,w}(X; p ,a)$ are very similar, so we only show $C_{2,w}(X; p ,a)$. We now treat $C_{2,w}(X; p, a)$. 

Evaluating the $\flat$ sum in \eqref{eq:chan} by the orthogonality relation \eqref{eq:genKloosa} for even characters, we have
\begin{align} \label{eq:645}
	C_{2,w}(X; p ,a)
	&=
	\frac{1}{p-1}
	\sum_{\substack{n < p}}
	\tau_2(n)
	\left(
	\frac{1}{2}
	(p-2)
	\overline{\Kl_2(an, p)}
	-1\right)
	\\& \quad \times
	\frac{1}{2\pi i}
	\int\displaylimits_{(1+\epsilon)}
	\mathcal{M}[w](1-s) 
	\frac{X^{1-s}}{n^s}
	\left(
	\pi^{\frac{1}{2} - s}
	p^{s-1}
	\frac{\Gamma\left(\frac{s}{2}\right)}{\Gamma\left(\frac{1-s}{2} \right)}
	\right)^2
	ds.
\end{align}
Contribution from the second term in the first parenthesis is
\begin{equation}
	\ll \frac{1}{p}
	\left( \frac{X}{p^2} \right)^{- \epsilon}
	\ll \frac{X^{1-\delta}}{p}.
\end{equation}
Thus, equation \eqref{eq:645} becomes
\begin{align} \label{eq:645c}
	\quad &C_{2,w}(X; p ,a)
	\\& \quad = 
	\frac{1}{2}
	\frac{p-2}{p-1}
	\frac{1}{2\pi i}
	\int\displaylimits_{(1+\epsilon)}
	\mathcal{M}[w](1-s) 
	\left(\frac{X}{p^2}\right)^{1-s}
	\left(
	\pi^{\frac{1}{2} - s}
	\frac{\Gamma\left(\frac{s}{2}\right)}{\Gamma\left(\frac{1-s}{2} \right)}
	\right)^2
	\sum_{\substack{n < p}}
	\frac{\tau_2(n) \Kl_2(an, p)}{n^s}
	ds
	\\ & \quad  \quad + O(p^{-1} X^{1- \delta}).
\end{align}
We now apply Weil's bound \eqref{eq:WeilBound} for Kloosterman sums, which gives
\begin{equation}
	C_{2,w}(X; p ,a)
	\ll
	\left( \frac{X}{p^2} \right)^{- \epsilon}
	p^{1/2 + \epsilon}
\end{equation}
and this is
\begin{equation}
	\ll \frac{1}{p} X^{1-\delta}
\end{equation}
when 
\begin{equation} \label{eq:knownlevel}
	p \ll X^{2/3 - \delta}.
\end{equation}
Similarly, using the condition \eqref{eq:orthogonalityodd} for odd primitive characters in place of \eqref{eq:orthogonalityeven}, we can also show that
\begin{equation}
	L_{2,w}(X; p ,a)
	\ll p^{-1} X^{1-\delta},
\end{equation}
for $p$ satisfying \eqref{eq:knownlevel}. Thus, by \eqref{eq:916} and the above two estimates, we obtain
\begin{align}
	\Delta_{2, w}(X; p ,a)
	\ll p^{-1} X^{1-\delta}
\end{align}
uniformly for $p$ satisfying \eqref{eq:knownlevel}. This leads to Theorem \ref{thm:1}.   \qed

\section{Proof of Proposition \ref{prop:onLH}}

Fix an $s = 1 +\epsilon + it$. By Weil's bound for Kloosterman sums, the function 
	\begin{equation}
		D_2(s+w; p, a)
		=
		\sum_{n =1}^\infty
		\frac{\tau_2(n) \Kl_2(an, p)}{n^{s+w}}
	\end{equation}
	is absolutely convergent for $\Re(w) > 0$. Thus, by Perron's formula and Weil's bound once more, we have
	\begin{equation} \label{eq:750}
		\sum_{n \le N}
		\frac{\tau_2(n) \Kl_2(an, p)}{n^s}
		=
		\frac{1}{2\pi i}
		\int\limits_{\epsilon - iT}^{\epsilon + iT}
		\frac{N^w}{w}
		D_2(s+w; p, a)
		dw
		+
		O \left( \frac{ p^{1/2} N^{\epsilon}}{T} \right),
	\end{equation}
	where $T$ is a parameter to be chosen; see \eqref{eq:T} below. By Lemma \ref{lem:D2ACFE}, the function $D_2(s+w; p, a)$ is analytic everywhere except for a double pole at $w = 1- s$. We note the residue at $w=0$:
	\begin{equation} \label{eq:751}
		D_2(s; p, a).
	\end{equation}
	This residue is not small in size but it is independent of the length $N$, and so will be subtracted off after we have evaluated the left side of \eqref{eq:750} and apply the result to $N=N$ and $N = N_1$, with $N_1<N$. This is the reason for the lower bound $n > N_1$ on the left side of \eqref{eq:750h}.
	
	Next, by \eqref{eq:739b}, we have
	\begin{align}
		D_2(s+w, ; p, a)
		&=
		\frac{1}{p-1} \zeta^2(s + w) 
		+
		\frac{p}{p-1} 
		\Star\sum_{\chi_1 \pmod p}
		\overline{\chi_1}(A)
		\tau(\chi_1)
		i^\aa \pi^{s + w - \frac{1}{2}}
		p^{s+ w -1}
		\\ & \quad \times
		\frac{\Gamma\left(\frac{1}{2}(1-s-w+\aa)\right)}{\Gamma\left(\frac{1}{2} (s+ w+\aa)\right)}
		L(1-s - w,\chi_1)
		L(s+ w, \overline{\chi_1}).
	\end{align}
	Inserting the above expression into \eqref{eq:750}, we shift the line of integration there left to the $-1/2 - \epsilon$ line (so that $\Re(s+w) = 1/2$), we pick up the residue \eqref{eq:751} at $w=0$, the residue at $w= 1- s$ of size $\ll N^{- \epsilon} (1 + |t|)^{-1} p^{-1 + \epsilon}$, by \eqref{eq:resofD2}, and two horizontal contours each of size 
	$\ll N^\epsilon T^{-1} p^{1/2 + \epsilon}
	+ N^{-1/2-\epsilon} T^{-1} p^{5/4 + \epsilon} T^{-B}$, by \eqref{eq:D2convexbdoncriticalline}, giving
	\begin{align} \label{eq:750c}
		\quad \quad
		&\sum_{n \le N}
		\frac{\tau_2(n) \Kl_2(an, p)}{n^s}
		=
		\sum_{n =1}^\infty
		\frac{\tau_2(n) \Kl_2(an, p)}{n^{s}}
		+ E_1
		\\&
		+ \frac{1}{2\pi i}
		\int\limits_{-1/2 - \epsilon - iT}^{-1/2 - \epsilon + iT}
		\frac{N^w}{w}
		\frac{p}{p-1} 
		\Star\sum_{\chi_1 \pmod p}
		\overline{\chi_1}(A)
		\tau(\chi_1)
		i^\aa \pi^{s + w - \frac{1}{2}}
		p^{s+ w -1}
		\frac{\Gamma\left(\frac{1}{2}(1-s-w+\aa)\right)}{\Gamma\left(\frac{1}{2} (s+ w+\aa)\right)}
		\\
		& \quad\quad \times
		L(1-s - w,\chi_1)
		L(s+ w, \overline{\chi_1})
		dw
		+
		O( T^{-1} p^{1/2+\epsilon} + p^{-1+\epsilon} ),
	\end{align}
	where
	\begin{equation}
		E_1
		=
		\frac{1}{p-1}
		\frac{1}{2\pi i}
		\int\limits_{-1/2 - \epsilon - iT}^{-1/2 - \epsilon + iT}
		\frac{N^w}{w}
		\zeta(s+w) dw
		\ll
		p^{-1} N^{-1/2 - \epsilon} T^{1/6},
	\end{equation}
	by Weyl's bound for $\zeta$.
	Next, by the approximate functional equation \eqref{eq:AFE}, we have
	\begin{equation}
		L(1-s - w,\chi_1)
		L(s+ w, \overline{\chi_1})
		=
		2 \sum_{m,n}
		\frac{\chi_1(m) \overline{\chi_1}(n)}{m^{1-s-w} n^{s+w}}
		V_\mathfrak{a} \left( \frac{mn}{p} \right).
	\end{equation}
	Splitting up the $\Star\sum$ sum in \eqref{eq:750c} into even $\Flat\sum$ and odd $\Sharp\sum$ sums, we have, by Lemma \ref{lem:13},
	\begin{equation} \label{eq:440}
		\Flat\sum_{\chi_1 \pmod p}
		\overline{\chi_1}(A) \tau(\chi_1)
		\chi_1(m) \overline{\chi_1}(n)
		=
		\frac{p-4}{2} e_p(A n \overline{m}) - 1.
	\end{equation}
	and
	\begin{equation}
		\Sharp\sum_{\chi_1 \pmod p}
		\overline{\chi_1}(A) \tau(\chi_1)
		\chi_1(m) \overline{\chi_1}(n)
		=
		\frac{p-1}{2} e_p(A n \overline{m}).
	\end{equation}
	Contribution from the $-1$ factor in \eqref{eq:440} to the integral in \eqref{eq:750c} is $\ll N^{-1/2 - \epsilon} T^{1/3}$,
	by Weyl's bound for $\zeta$ once more.
	Hence, \eqref{eq:750c} becomes
	\begin{align} \label{eq:750d}
		&\sum_{n \le N}
		\frac{\tau_2(n) \Kl_2(an, p)}{n^s}
		=
		\sum_{n =1}^\infty
		\frac{\tau_2(n) \Kl_2(an, p)}{n^{s}}
		\\& \quad
		+ \frac{1}{2\pi i}
		\int\limits_{-1/2 - \epsilon - iT}^{-1/2 - \epsilon + iT}
		\frac{N^w}{w}
		\frac{p}{p-1} 
		\pi^{s + w - \frac{1}{2}}
		p^{s+ w -1}
		\frac{\Gamma\left(\frac{1}{2}(1-s-w)\right)}{\Gamma\left(\frac{1}{2} (s+ w)\right)}
		\\
		& \quad\quad \times
		(p-4) 
		\sum_{m,n}
		\frac{ e_p(A n \overline{m})}{m^{1-s-w} n^{s+w}}
		V_0 \left( \frac{mn}{p} \right)
		dw
		\\& \quad
		+ \frac{1}{2\pi i}
		\int\limits_{-1/2 - \epsilon - iT}^{-1/2 - \epsilon + iT}
		\frac{N^w}{w}
		\frac{p}{p-1} 
		i \pi^{s + w - \frac{1}{2}}
		p^{s+ w -1}
		\frac{\Gamma\left(\frac{1}{2}(2-s-w)\right)}{\Gamma\left(\frac{1}{2} (s+ w+1)\right)}
		\\
		& \quad\quad \times
		(p-1) 
		\sum_{m,n}
		\frac{ e_p(A n \overline{m})}{m^{1-s-w} n^{s+w}}
		V_1 \left( \frac{mn}{p} \right)
		dw
		\\& \quad \quad \quad
		+
		O \left( 
		T^{-1} p^{1/2+\epsilon} + p^{-1+\epsilon}
		+ p^{-1} N^{-1/2 - \epsilon} T^{1/6}
		+ N^{-1/2 - \epsilon} T^{1/3}
		\right).
	\end{align}
	Shifting the line of integration in \eqref{eq:AFE2} to the $(\epsilon)$ line, we get, by Conjecture \ref{conj:LH},
	\begin{equation} \label{eq:454}
		\sum_{m,n}
		\frac{ e_p(A n \overline{m})}{m^{1-s-w} n^{s+w}}
		V_\mathfrak{a} \left( \frac{mn}{p} \right)
		\ll 
		p^\epsilon 
		(1 + |s+w| )^\epsilon,
		\quad
		(\aa = 0,1).
	\end{equation}
	With this, equation \eqref{eq:750d} becomes
	\begin{align} \label{eq:750e}
		&\sum_{n \le N}
		\frac{\tau_2(n) \Kl_2(an, p)}{n^s}
		=
		\sum_{n =1}^\infty
		\frac{\tau_2(n) \Kl_2(an, p)}{n^{s}}
		+ O \left( N^{-1/2 - \epsilon} p^{-1/2} p^{1+\epsilon} \right)
		\\& \quad
		+ O \left( 
		T^{-1} p^{1/2+\epsilon} + p^{-1+\epsilon}
		+ p^{-1} N^{-1/2 - \epsilon} T^{1/6}
		+ N^{-1/2 - \epsilon} T^{1/3}
		\right).
	\end{align}
	Taking
	\begin{equation} \label{eq:T}
		T = N^{1/2},
	\end{equation}
	we obtain, from \eqref{eq:750e},
	\begin{align} \label{eq:750f}
		&\sum_{n \le N}
		\frac{\tau_2(n) \Kl_2(an, p)}{n^s}
		=
		\sum_{n =1}^\infty
		\frac{\tau_2(n) \Kl_2(an, p)}{n^{s}}
		+ O \left( N^{-1/2} p^{1/2+\epsilon} \right),
	\end{align}
	on assuming Conjecture \ref{conj:LH} to bound \eqref{eq:454}. Applying \eqref{eq:750f} to $N=p$ and $N=N_1$, we obtain
	\begin{align} \label{eq:750g}
		&\sum_{N_1 \le n \le N}
		\frac{\tau_2(n) \Kl_2(an, p)}{n^s}
		\ll 
		N^{-1/2} p^{1/2+\epsilon}
		+ N_{1}^{-1/2} p^{1/2+\epsilon}.
	\end{align}
	This leads to \eqref{eq:750h}. \qed

\section{Proof of Theorem \ref{thm:beatsSelberg}}
Assume Conjecture \ref{conj:LH} for prime moduli $p$. This assumption is solely for the proof of Proposition \ref{prop:onLH} and the rest of the arguments in this section is unconditional.

By Theorem \ref{thm:1}, we may henceforth assume that
\begin{equation} \label{eq:knownlevellb}
	p \gg X^{2/3 - \delta}.
\end{equation}
We use this lower bound \eqref{eq:knownlevellb} to truncate the $n$ sum in \eqref{eq:645} a bit more. This is done for technical reasons to eliminate the residue \eqref{eq:751} at zero. Let $N_1 \ge 1$ be a parameter to be determined shortly. Shifting the line of integration in \eqref{eq:645c} left to the line $\sigma = \epsilon$, contribution from terms $n \le N_1$ in this quantity \eqref{eq:645c} is
\begin{equation} \label{eq:204}
	\ll 
	\left( \frac{X N_1}{p^2} \right)^{1 - \epsilon}
	p^{1/2} \log(N_1),
\end{equation}
which is $\ll p^{-1} X^{1-\delta}$ when we take, for instance,
\begin{equation} \label{eq:N1}
	N_1 
	=
	p^{1/2 - \delta}.
\end{equation}
Note that, by \eqref{eq:knownlevellb},
\begin{equation}
	N_1 \gg X^{1/3 - 2 \delta}.
\end{equation}
Hence, with this choice \eqref{eq:N1} for $N_1$, equation \eqref{eq:645} becomes
\begin{align} \label{eq:645d}
	\quad &C_{2,w}(X; p ,a)
	\\& \quad = 
	\frac{1}{2}
	\frac{p-2}{p-1}
	\frac{1}{2\pi i}
	\int\displaylimits_{(1+\epsilon)}
	\mathcal{M}[w](1-s) 
	\left(\frac{X}{p^2}\right)^{1-s}
	\left(
	\pi^{\frac{1}{2} - s}
	\frac{\Gamma\left(\frac{s}{2}\right)}{\Gamma\left(\frac{1-s}{2} \right)}
	\right)^2
	\sum_{\substack{p^{\frac{1}{2} - \delta} < n < p}}
	\frac{\tau_2(n) \Kl_2(an, p)}{n^s}
	ds
	\\ & \quad  \quad 
	+ O(p^{-1} X^{1-\delta}).
\end{align}
We treat the $n$ summation in the above using Proposition \ref{prop:onLH}, giving
\begin{equation}
	\sum_{\substack{p^{\frac{1}{2} - \delta} < n < p}}
	\frac{\tau_2(n) \Kl_2(an, p)}{n^s}
	\ll
	p^{1/2+\epsilon}
	\left( p^{1/2-\delta} \right)^{-1/2}
	\ll p^{\frac{1}{4} + \delta}.
\end{equation}
With this, the integral on the right side of \eqref{eq:645d} is bounded by
\begin{equation}
	O(p^{1/4 + 2 \delta})
\end{equation}
and this is $\ll p^{-1} X^{1-\delta}$ when
\begin{equation} \label{eq:712}
	p \ll X^{\frac{4}{5} - 2 \delta}.
\end{equation}
Thus, from \eqref{eq:645d},
\begin{equation}
	C_{2,w} \ll p^{-1} X^{1-\delta},
\end{equation}
for all $p$ satisfying \eqref{eq:712}.

Similarly, using the condition \eqref{eq:orthogonalityodd} for odd primitive characters in place of \eqref{eq:orthogonalityeven}, we can also show that
\begin{equation}
	L_{2,w}(X; p ,a)
	\ll p^{-1} X^{1-\delta},
\end{equation}
for $p$ satisfying \eqref{eq:712}. Thus, by \eqref{eq:916} and the above two estimates, we obtain
\begin{align}
	\Delta_{2, w}(X; p ,a)
	\ll p^{-1} X^{1-\delta}
\end{align}
uniformly for all $p$ satisfying \eqref{eq:712}, on assuming Conjecture \ref{conj:LH}. This leads to Theorem \ref{thm:beatsSelberg}.   \qed

\section{Heuristics for Conjecture \ref{conj:LH}} \label{sec:Heuristics}

In this last section, we give a heuristic argument based on a twisted second moment estimate which leads to Conjecture \ref{conj:LH}. It is best read with a physics hat on.

By adding an extra variable, we may move the critical strip right to the region of absolute convergence so we can interchange order of summations and evaluate term by terms. This is the basic strategy. As such, the extra variable $w$ in \eqref{eq:AFE2} does that, but for more flexibility, we separate variables in there, so that
\begin{equation}
	V_\mathfrak{a} \left( \frac{nm}{q} \right)
	\rightsquigarrow
	N_\mathfrak{a} \left( \frac{n}{q} \right)
	M_\mathfrak{a} \left( \frac{m}{q} \right),
\end{equation}
say, with $N_\mathfrak{a}$ and $M_\mathfrak{a}$ having similar meanings as {\eqref{eq:AFE2}} (there are some flexibility in choosing the quenching poles-remover factor $e^{w^2} \cos^2(\pi w)$). Doing this, \eqref{eq:AFE2} transforms into
\begin{equation} \label{eq:newbeefb}
	\sum_{\substack{n,m=1\\ (mn,q)=1}}^\infty
	\frac{e_q (A n \overline{m})}{n^s m^{1-s}}
	N_\mathfrak{a} \left( \frac{n}{q} \right)
	M_\mathfrak{a} \left( \frac{m}{q} \right), \quad
	(s=1/2 + it).
\end{equation}
Now move $n^{-w}$ in $N_\mathfrak{a}$ and $m^{-z}$ in $M_\mathfrak{a}$ right so that both $\Re(s+w)$ and $\Re(1-s+z)$ are larger than 1. Term by term, for each fixed $m$, a result of Estermann \cite{Estermann1930} says that
\begin{equation} \label{eq:1158}
	\sum_{n=1}^{\infty}
	\frac{e_q(A \overline{m} n)}{n^{s+w}}
	=
	G(s) q^{1-s-w}
	\left[ e^{\pi i (s+w)/2} \zeta(1-s-w; -A \overline{m}, q) - e^{-\pi i (s+w)/2} \zeta(1-s-w; A \overline{m}, q) \right],
\end{equation}
where
\begin{equation}
	\zeta(s; a, q) =
	\sum_{n \equiv a \pmod q}
	\frac{1}{n^s}, \quad (\sigma >1)
\end{equation}
and
\begin{equation}
	G(s) = - i (2 \pi)^{s-1} \Gamma(1-s).
\end{equation}
Note that, for $\Re(s+w)=1+\epsilon$, we have $q^{1-s-w} \ll q^{-\epsilon}$. Insert the expression \eqref{eq:1158} into \eqref{eq:newbeefb}, then move the $w$-line left so that $\Re(1-s-w) >1$. This gives
\begin{align}
	&G(s) q^{1-s-w}
	\left[ e^{\pi i (s+w)/2} 
	\sum_{(m,q)=1} \frac{1}{m^{1-s +z}}
	\zeta(1-s-w; -A \overline{m}, q) 
	\right.
	\\ & \left. \quad \quad
	- e^{-\pi i (s+w)/2} 
	\sum_{(m,q)=1} \frac{1}{m^{1-s+z}}
	\zeta(1-s-w; A \overline{m}, q) \right],
	\quad
	(\Re(1-s-w), \Re(1-s+z) >1).
\end{align}
Note that $\Re(1-s+z)$ is still larger than 1. Going to Dirichlet characters as in \eqref{eq:delta}, the first $m$ sum on the right side of the above is equal to
\begin{align}
	&\sum_{(m,q)=1} \frac{1}{m^{1-s+z}}
	\frac{1}{\varphi(q)}
	\sum_{\chi \pmod q}
	\chi(-A \overline{m}) 
	\sum_{n=1}^\infty
	\frac{\overline{\chi}(n)}{n^{1-s-w}}
	\\ & \quad
	=
	\frac{1}{\varphi(q)}
	\sum_{\chi \pmod q}
	\chi(-A)
	L(1-s+z, \overline{\chi})
	L(1-s-w, \overline{\chi}).
\end{align}
The twisted second moment on the right side of the above is
\begin{equation}
	\ll q^{\epsilon} (1 + |\Im(1+s+z+w)|)^\epsilon
\end{equation}
on the critical line. Thus, by the above bound, and ignoring all error terms, we have, for fixed imaginary parts on the critical line,
\begin{equation} \label{eq:625b}
	\sum_{\substack{n,m=1\\ (mn,q)=1}}^\infty
	\frac{e_q (A n \overline{m})}{n^s m^{1-s}}
	N_\mathfrak{a} \left( \frac{n}{q} \right)
	M_\mathfrak{a} \left( \frac{m}{q} \right)
	\ll
	q^{\epsilon}.
\end{equation}
This leads to Conjecture \ref{conj:LH}.
\hfill "Q.E.D."

\bibliographystyle{plain}

\end{document}